\documentclass[11pt]{article}

\usepackage{amsfonts}
\usepackage{amsmath}

\newcommand{\beq}{\begin{equation}}
\newcommand{\eeq}{\end{equation}}
\newcommand{\bea}{\begin{eqnarray}}
\newcommand{\eea}{\end{eqnarray}}
\newcommand{\beas}{\begin{eqnarray*}}
\newcommand{\eeas}{\end{eqnarray*}}

\newtheorem{theorem}{Theorem}[section]

\newtheorem{definition}[theorem]{Definition}
\newtheorem{proposition}[theorem]{Proposition}

\newtheorem{remark}[theorem]{Remark}
\newtheorem{example}[theorem]{Example}
\newtheorem{examples}[theorem]{Examples}
\newtheorem{foo}[theorem]{Remarks}

\newenvironment{proof}{\addvspace{\medskipamount}\par\noindent{\it
Proof}.}
{\unskip\nobreak\hfill$\Box$\par\addvspace{\medskipamount}}

\title{Stochastic differential equations driven by loops}

\author{Fabrice Baudoin}

\date{Department of Mathematics, Purdue University \\
 West Lafayette, IN, USA}
 
\begin{document}
 \maketitle
 
\begin{abstract}
We study stochastic differential equations of the type
\begin{equation*}
X_t =x + \sum_{i=1}^d \int_0^t V_i (X_s^x) \circ dM^i_s, \text{ }
0 \leq t \leq T,
\end{equation*}
where $(M_s)_{0 \leq s \leq T}$ is a semimartingale generating a
loop in the free Carnot group of step $N$ and show how the
properties of the random variable $X_T^x$ are closely related to
the Lie subalgebra generated by the commutators of the $V_i$'s
with length greater than $N+1$. It is furthermore shown that if
$f$ is a smooth function, then
\[
\lim_{T\rightarrow
0}\frac{\mathbb{E}(f(X_T^x))-f(x)}{T^{N+1}}=(\Delta _{N}f)(x),
\]
where $\Delta _{N}$ is a second order operator related to the
$V_i's$.

\


\end{abstract}

\maketitle

\baselineskip 0.25in

\tableofcontents

\section{Introduction}

Let us consider a stochastic differential equations on
$\mathbb{R}^n$ on the type
\begin{equation}  \label{SDE intro}
X_t^{x_0} =x_0 + \sum_{i=1}^d \int_0^t V_i (X_s^{x_0}) \circ
dM^i_s, \text{ } 0 \leq t \leq T,
\end{equation}
where:

\begin{enumerate}
\item  $x_{0}\in \mathbb{R}^{n}$;

\item  $V_{1},...,V_{d}$ are $C^{\infty }$ bounded vector fields on $\mathbb{R}^{n}$;

\item  $\circ $ denotes Stratonovitch integration;

\item  $(M_t)_{0\leq t\leq T}=(M_{t}^{1},...,M_{t}^{d})_{0\leq t\leq T}$ is a
$d$-dimensional continuous semimartingale.
\end{enumerate}
It is well-known that if $(M_t)_{0\leq t\leq T}$ is a standard
Brownian motion then for any smooth function
$f:\mathbb{R}^{n}\rightarrow \mathbb{R}$ we have in
$\mathbf{L}^{2}$,
\[
\lim_{t\rightarrow 0}\frac{\mathbf{P}_{t}f-f}{t}=\frac{1}{2}
\left( \sum_{i=1}^d V_i^2 \right) f,
\]
where $\mathbf{P}_{t}$ is the semigroup associated with (\ref{SDE
intro}) which is defined by
\[
\mathbf{P}_{t}f(x)=\mathbb{E}\left( f(X_{t}^{x})\right).
\]
In that case, it is furthermore known from H\"ormander's theorem
that $\mathbf{P}_{t}$ has a smooth transition kernel with respect
to the Lebesgue measure as soon as for all $x_{0}\in
\mathbb{R}^{n}$ , $\mathfrak{L} (x_0)=\mathbb{R}^{n}$ where
$\mathfrak{L}$ is the Lie algebra generated by the vector fields
$V_i$'s. So, when $(M_t)_{0\leq t\leq T}$ is a Brownian motion,
the properties of $(X_t^{x_0})_{0 \leq t \leq T}$ are closely
related to the diffusion operator $ \sum_{i=1}^d V_i^2 $ and the Lie algebra $\mathfrak{L}$.

In this paper, we show  that there are other choices of  the driving  semimartingale $(M_t)_{0\leq t\leq T}$ for which the solution $(X_t^{x_0})_{t \ge 0}$ is naturally associated to other diffusion operators.

Let us roughly describe our approach. The Chen-Strichartz
expansion theorem (see \cite{Baubook}, \cite{Che}, \cite{Stri})
states that, formally, the stochastic flow $(\Phi_t)_{0 \leq t
\leq T}$ associated with the stochastic differential equations
(\ref{SDE intro}) can be written as 
\[
\Phi_{t}=\exp\left(
\sum_{k\geq1}\sum_{i_{1},...,i_{k}\in\{1,...,d\}}
F_{i_{1},...,i_{k}}\left( \int_{0\leq t_{1}\leq\dots\leq t_{k}\leq
t}\circ dM_{t_{1}}^{i_{1}}\dots\circ dM_{t_{k}}^{i_{k}}\right)
_{anti}\right) , \text{ } t \leq T,
\]
where $F_{i_{1},...,i_{k}}$ are universal Lie polynomials in
$V_{1},...,V_{d} $, which depend on the choice of a Hall basis in
the free Lie algebra with $d$ generators, and $\left( \int_{0\leq
t_{1}\leq\dots\leq t_{k}\leq t}\circ dM_{t_{1} }^{i_{1}}\dots\circ
dM_{t_{k}}^{i_{k}}\right) _{anti}$ are universal antisymmetrizations of the iterated integrals of the semimartingale $(M_t)_{t \ge 0}$.

Consider now $N \geq 0$, and take for $(M^1_{t},...,M^d_{t})_{ 0
\leq t \leq T}$ a $d$-dimensional standard Brownian motion
conditioned by
\[
\left( \int_{0\leq t_{1}\leq\dots\leq t_{k}\leq T}\circ dM_{t_{1}
}^{i_{1}}\dots\circ dM_{t_{k}}^{i_{k}}\right) _{anti}=0, \text{ }
i_{1},...,i_{k}\in\{1,...,d\},\text{ }1 \leq k \leq N.
\]
It is shown that such a process, that we call a $N$-step Brownian
loop, is a semimartingale up to time $T$ and can be constructed
from a diffusion in the loop space over the free Carnot group of
step $N$.

For this choice of $(M_t)_{t \ge 0}$,  the Chen development for $\Phi_T$ writes
\[
\Phi_{T}=\exp\left(
\sum_{k\geq N+1}\sum_{i_{1},...,i_{k}\in\{1,...,d\}}
F_{i_{1},...,i_{k}}\left( \int_{0\leq t_{1}\leq\dots\leq t_{k}\leq
t}\circ dM_{t_{1}}^{i_{1}}\dots\circ dM_{t_{k}}^{i_{k}}\right)
_{anti}\right),
\]
and thus only involves the
Lie subalgebra $\mathfrak{L}^{N+1}$, where $\mathfrak{L}$ is the
Lie algebra generated by the vector fields $V_i$'s and for $p \geq
2$, $\mathfrak{L}^p$ is inductively defined by
\[
\mathfrak{L}^p = \{ [X,Y], \text{ } X \in \mathfrak{L}^{p-1}, Y \in %
\mathfrak{L} \},\text{  }\mathfrak{L}^1=\mathfrak{L}.
\]
Hence, we can expect that the properties of the random variable
$X_T^x$, where $(X_t^x)_{0 \leq t \leq T}$ is the solution of
(\ref{SDE intro}) with initial condition $x$, are closely related
to this Lie subalgebra $ \mathfrak{L}^{N+1}$.

Precisely, we show that:

\begin{itemize}

\item  If $f:\mathbb{R}^{n}\rightarrow \mathbb{R}$ is a smooth function which is compactly supported,
then in $L^{2}$,
\[
\lim_{T\rightarrow 0}\frac{\mathcal{H}_{T}^{N}f-f}{T^{N+1}}=\Delta
_{N}f,
\]
where $\Delta _{N}$ is an homogeneous second order differential
operator that belongs to the universal enveloping algebra of
$\mathfrak{L}^{N+1}$ and $\mathcal{H}_{T}^{N}$ the $N$-step
holonomy operator that we define by
\[
\mathcal{H}_{T}^{N+1}f(x)=\mathbb{E}\left( f(X_{T}^{x})\right).
\]

\item  If $\mathfrak{L}^{N+1}=0$, then for all $x\in \mathbb{R}^{n}$, almost
surely $X_{T}^{x}=x$;

\item  If for all $x\in \mathbb{R}^{n}$, $\mathfrak{L}^{N+1}(x)=\mathbb{R}
^{n}$, then for all $x\in \mathbb{R}^{n}$, $X_{T}^{x}$ has a
smooth density $ p_{T}(x)$ with respect to the Lebesgue measure.
Moreover, in that case $ p_{T}(x)\sim _{T\rightarrow
0}m(x)T^{-\frac{D(x)}{2}}$, where $m$ is a smooth non negative
function and $D(x)$ an integer (the graded dimension of $
\mathfrak{L}^{N+1}(x)$);

\end{itemize}
We stress the fact that from a geometrical point of view the
family of operators $(\Delta_N)_{N \geq 0}$ is quite interesting
and is an important invariant of the intrinsic geometry of the
differential system generated by the $V_i$'s. For instance we
shall see that, for some positive constant $C$,
\[
\Delta_1= C \sum_{i,j=1}^d [V_i ,V_j ]^2
\]
and, in a way, this operator sharply measures the curvature of the
differential system generated by the vector fields $V_i$.

Let us mention that, since the seminal work of Rotschild and Stein
\cite{Ro-St}, the study  Carnot groups arise naturally in PDE's
theory and is now an active research field . Our interest in SDE's
driven by loops in Carnot groups originally come from the study of
the Brownian holonomy on sub-Riemannian manifolds. The
understanding of this holonomy is closely related to the
construction of parametrices for hypoelliptic Schr\"odinger
equations (see \cite{Ba3}).

The paper is organized as follows. The second section is here for
the sake of clarity of the paper since the framework is quite
simple but the results already interesting: We study stochastic
differential equations driven by Brownian loops. All the results
presented in this section will be later generalized. In the third
section, we introduce the notion of free Carnot group of step $N$
and define a fundamental diffusion on it. We then study the
coupling of this diffusion with the solution of a generic
stochastic differential equation driven by Brownian motions. In
particular we establish an H\"{o}rmander type theorem for the
existence of a smooth density for the joint law of this coupling.
The fourth section constitutes the heart of this paper and gives
the proofs of the results presented above.

\

Some results of the paper were already  announced in the note \cite{B1} and the book \cite{Baubook}.

\section{Stochastic Differential Equations driven by Brownian Loops and Bridges}

We consider first on $\mathbb{R}^n$ stochastic differential
equations of the type
\begin{equation}  \label{SDEbridge}
X_t =x_0 + \sum_{i=1}^d \int_0^t V_i (X_s) \circ dP^i_{s,T},
\text{ } t \leq T
\end{equation}
where:

\begin{enumerate}
\item  $x_0 \in \mathbb{R}^n$;

\item  $V_1,...,V_d$ are $C^{\infty}$ bounded vector fields on $\mathbb{R}^n$%
;

\item  $(P^1_{t,T},...,P^d_{t,T})_{0 \leq t \leq T}$ is a given $d$
-dimensional Brownian bridge from $0$ to $0$ with length $T>0$.
\end{enumerate}

Notice that since $(P_{t,T})_{0 \leq t \leq T}$ is known to be a
semimartingale up to time $T$, the notion of solution for
(\ref{SDEbridge}) is well-defined up to time $T$.

\begin{proposition}
For every $x_0 \in \mathbb{R}^n$, there is a unique solution
$(X_t^{x_0})_{0 \leq t \leq T}$ to (\ref{SDEbridge}). Moreover
there exists a stochastic flow $(\Phi_t, 0 \leq t \leq T)$ of
smooth diffeomorphisms $\mathbb{R}^n \rightarrow \mathbb{R}^n$
associated to the equations (\ref{SDEbridge}).
\end{proposition}

\begin{proof}
We refer to the book of Kunita \cite{Kun}, where the questions of
existence and uniqueness of a smooth flow for stochastic
differential equations driven by general continuous
semimartingales are treated (cf. Theorem 3.4.1. p. 101 and Theorem
4.6.5. p. 173).
\end{proof}

The random variable $X_T$ where $(X_t)_{0 \leq t \leq T}$ is a
solution of (\ref{SDEbridge}), is of particular interest: It is
closely related to the commutations properties of the vector
fields $V_1,...,V_d$. Indeed, let us consider the following family
of operators $(\mathcal{H}_T)_{T \geq 0}$ defined on the space of
compactly supported smooth functions $f:\mathbb{R}^n \rightarrow
\mathbb{R}$ by
\[
(\mathcal{H}_T f)(x) =\mathbb{E}\left( f(X_{T}^{x})\right), \text{
}x \in \mathbb{R}^n.
\]

Obviously, the family of operators $(\mathcal{H}_T)_{T \geq 0}$ does not
satisfy the semigroup property. It is interesting that in some cases, we can explicitly compute $(\mathcal{H}_T)_{T \geq 0}$.

\begin{theorem}
Assume that the Lie algebra generated by the vector fields $V_i$
is two-step nilpotent (that is, any commutator with length greater
than 3 is 0) then
\[
\mathcal{H}_T=\det \left( \frac{T \Omega}{2 \sinh (\frac{1}{2} T \Omega)}\right)^{\frac{1}{2}},
\]
where $\Omega$ is the $d \times d$ matrix such that
$\Omega_{i,j}=[V_i,V_j]$.
\end{theorem}

Before turning to the proof, we mention that the above expression for $\mathcal{H}_T$ is understood in the
sense of pseudo-differential operators. Namely, the expression
\[
\det \left( \frac{\left(x_{i,j}\right)_{1\leq i,j \leq d}}{2 \sinh
(\frac{1}{2} \left(x_{i,j}\right)_{1\leq i,j \leq
d})}\right)^{\frac{1}{2}}
\]
defines an analytic function
\[
\Phi(\left(x_{i,j}\right)_{1\leq i,j \leq d})
\]
and the above theorem says that
\[
\mathcal{H}_T=\int_{\mathbb{R}^{\frac{d(d-1)}{2}}} \hat{\Phi}
(\xi) e^{i T \sum_{i<j} \xi_{i,j} [V_i,V_j]} d\xi
\]
where $\hat{\Phi}$ denotes the Fourier transform of $\Phi$. For
further details on pseudo-differential operators we refer to the
chapter 7 of \cite{Tay2}.

\begin{proof}
It\^o's formula shows that in that two-nilpotent case,
\[
f(X_T^x)=\left( \exp \left( \frac{1}{2} \sum_{1 \leq i<j \leq d}
[V_i,V_j] \int_0^t P^i_{s,T} dP^j_{s,T}-P^j_{s,T} dP^i_{s,T}
\right) f \right)(x).
\]

But, from Gaveau-L\'evy's area formula see \cite{Ga}, if $A$ is a
$d \times d$ skew-symmetric matrix valued in a commutative ring,
then,
\begin{equation*}
\mathbb{E} \left( e^{i \int_0^T (A P_{s,T} ,dP_{s,T} )} \right)=
\det \left( \frac{tA}{\sin tA} \right)^{\frac{1}{2}}.
\end{equation*}
This completes the proof.
\end{proof}
It seems difficult to find a closed expression for
$\mathcal{H}_{T}$ in the general case, we can nevertheless compute a small-time asymptotics:
\begin{theorem}
Let $f:\mathbb{R}^n \rightarrow \mathbb{R}$ be a smooth function
which is compactly supported. In $\mathbf{L}^2$,
\[
\lim_{T\rightarrow
0}\frac{\mathcal{H}_{T}f-f}{T^{2}}=\frac{1}{24}\left( \sum_{1 \leq
i<j \leq d} [V_i,V_j]^2 \right) f.
\]
\end{theorem}
\begin{proof}
We refer to the  proof of Theorem \ref{Theogeneral} which is more general. We however show how the constant $\frac{1}{24}$ is obtained.
The proof of Theorem \ref{Theogeneral} shows that there is a
universal constant $C$ such that
\[
\lim_{T\rightarrow
0}\frac{\mathcal{H}_{T}f-f}{T^{2}}=C\left( \sum_{1 \leq
i<j \leq d} [V_i,V_j]^2 \right) .
\]
Since this constant is universal, in order to compute it, it
suffices to look at the two-step nilpotent case. In that case,
from the previous theorem
\[
\mathcal{H}_T=\det \left( \frac{T \Omega}{2 \sinh (\frac{1}{2} T
\Omega)}\right)^{\frac{1}{2}}.
\]
Therefore
\[
\mathcal{H}_T\sim_{T \rightarrow 0} \det \left( \mathbf{1}-\frac{1}{24}T^2 \Omega^2 \right)^{\frac{1}{2}},
\]
and the computation is easily done.
\end{proof}
We study now sufficient conditions which ensure that the operator
$\mathcal{H}_{T}$ has a smooth kernel (in the two variables) with
respect to the Lebesgue measure of $\mathbb{R}^n$. To answer this
question, it is enough to decide under which conditions the random
variable $X_{T}^{x}$ has a smooth density.

On one hand, we have the following result:

\begin{theorem}
\label{Holonomie step2} Assume that
$[\mathfrak{L},\mathfrak{L}]=0$, then for any solution
$(X_t^{x_0})_{0 \leq t \leq T}$ of (\ref{SDEbridge}) we have
almost surely $X_T^{x_0} = x_0$.
\end{theorem}

\begin{proof}
For $i=1,...,d$, let us denote $(e^{tV_i})_{t \in \mathbb{R}}$ the
one-parameter flow associated with the complete vector field
$V_i$. Since the $V_i$'s are assumed to commute, an iterative
application of It\^o's formula shows that the process
\begin{equation*}
\left( \left( e^{P^1_{s,T} V_1} \circ ... \circ e^{P^d_{s,T} V_d}
\right) \left( x_0 \right) \right)_{0 \leq s \leq T}
\end{equation*}
solves the equation (\ref{SDEbridge}) with initial condition
$x_0$. By uniqueness, we have hence
\begin{equation*}
\Phi_s (x_0)=  \left( e^{P^1_{s,T} V_1} \circ ... \circ
e^{P^d_{s,T} V_d} \right) (x_0) , \text{ }0 \leq s \leq T.
\end{equation*}
In particular,
\begin{equation*}
\Phi_T (x_0) =x_0,
\end{equation*}
which is the expected result.
\end{proof}

In general, the weaker condition
$[\mathfrak{L},\mathfrak{L}](x_0)=0$ is not
enough to conclude that for the solution $(X_t^{x_0})_{0 \leq t \leq T}$ of (%
\ref{SDEbridge}) we have almost surely $X_T^{x_0} = x_0$. For
instance, consider in dimension 2,
\[
V_{1}=\left(
\begin{array}{l}
1 \\
0
\end{array}
\right) ,\text{ and }V_{2}=\left(
\begin{array}{l}
0 \\
f\left( x\right),
\end{array}
\right)
\]
where $f$ is a smooth function whose Taylor development at 0 is 0
(by e.g. $ f(x)=e^{-\frac{1}{x^2}} \mathbf{1}_{x>0}$).
Nevertheless, if the vector fields $V_i$'s are assumed to be
analytic on whole $\mathbb{R}^n$, $[\mathfrak{L},
\mathfrak{L}](x_0)=0$ implies that $[\mathfrak{L},
\mathfrak{L}]=0$ and therefore that almost surely $X_T^{x_0} =
x_0$.

\begin{theorem}
\label{Hormander step2} Assume that
$[\mathfrak{L},\mathfrak{L}](x_0)= \mathbb{R}^n$, then for the
solution $(X_t^{x_0})_{0 \leq t \leq T}$ of (\ref {SDEbridge}) the
random variable $X_T^{x_0}$ has a smooth density with respect to
the Lebesgue measure of $\mathbb{R}^n$.
\end{theorem}

\begin{proof}
Let us consider  the solution $(Y_t)_{t \geq 0}$ of the following
stochastic differential equation:
\begin{equation*}
Y_t =x_0 + \sum_{i=1}^d \int_0^t V_i (Y_s) \circ dB^i_s, \text{ }
t \geq 0,
\end{equation*}
where $(B^1_{t},...,B^d_{t})_{ t \geq 0}$ is a $d$-dimensional
standard  Brownian motion. Since
$[\mathfrak{L},\mathfrak{L}](x_0)= \mathbb{R}^n$, it easily seen
that $(Y_t,B_t)_{t \geq 0}$ is a diffusion process whose
infinitesimal generator satisfies the (strong) H\"ormander's
condition at $(x_0,0)$. Therefore, the random variable
\begin{equation*}
(Y_T,B_T)
\end{equation*}
has a smooth density with respect to the Lebesgue measure on
$\mathbb{R}^n \times \mathbb{R}^n$. This implies the existence of
a smooth function $p:\mathbb{R}^n \rightarrow \mathbb{R}$ such
that for all bounded measurable function $f:\mathbb{R}^n
\rightarrow \mathbb{R}$
\begin{equation*}
\mathbb{E} (f(Y_T) \mid B_T=0)=\int_{\mathbb{R}^n} f (y) p(y) dy.
\end{equation*}
Now, since in law the process $(P_{t,T})_{0 \leq t \leq T}$, is
identical to the Brownian motion $(B_t)_{0 \leq t \leq T}$
conditioned by $B_T=0$, the function $p$ is actually exactly the
density of the random variable $X_T^{x_0}$ where $(X_t^{x_0})_{0
\leq t \leq T}$ is the solution of (\ref{SDEbridge}) with initial
condition $x_0$.
\end{proof}

Another proof of this result may be given by using standard Malliavin calculus tools (see chapter 2 of Nualart's book \cite{Nu}, whose notations below are
taken). 

We work  in the $d$-dimensional Wiener space and define
the Brownian loop $(P^1_{t,T},...,P^d_{t,T})_{0 \leq t \leq T}$ as
the Wiener integral
\[
P_{t,T} = (T-t) \int_0^t \frac{dW_s}{T-s}, \text{ }t<T, \text{ and
}P_{T,T} = 0,
\]
where $W$ is the $d$-dimensional Wiener process. In this setting,
it is not difficult to prove that if $(X_t)_{0 \leq t \leq T}$ is
a solution of (\ref {SDEbridge}), then $X_T \in
\mathbb{D}^{\infty}$. Moreover, a direct computation shows that
for any $0 \leq s \leq T$, the Malliavin derivative is given
\[
\mathbf{D}_s X_T=\mathbf{J}_{0 \rightarrow T} \left( \mathbf{J}_{0
\rightarrow s}^{-1} \sigma (X_s) - \frac{1}{T-s} \int_s^T
\mathbf{J}_{0 \rightarrow u}^{-1} \sigma (X_u) du \right),
\]
where $(\mathbf{J}_{0 \rightarrow t})_{ 0 \leq t \leq T}$ is the
first variation process defined by
\[
\mathbf{J}_{0 \rightarrow t}=\frac{\partial \Phi_t}{\partial x},
\]
and $\sigma$ the $n \times d$ matrix field $\sigma =
(V_1,...,V_d)$. From this, we can deduce that the Malliavin matrix
of $X_T$ must be invertible. Indeed, if not, we could find a non
zero vector $h \in \mathbb{R}^d$ and a finite stopping time
$\theta >0$ such that $\mathbf{D}_s X_T \cdot h =0$ for $0 \leq s
\leq \theta$. This would lead to the conclusion that
$(\mathbf{J}_{0 \rightarrow s}^{-1} \sigma (X_s) \cdot h)_{0 \leq
s \leq T}$ must be constant.

Observe now that
\[
\mathbf{J}_{0 \rightarrow s}^{-1} V_i (X_s) =\Phi_s^{\ast} V_i,
\]
where $\Phi$ denotes the stochastic flow associated with equation
(\ref{SDEbridge}), and where $\Phi_s^{\ast} V_i$ denotes the
pull-back action of $\Phi$ on $V_i$. Therefore, according to the
It\^o's formula, we obtain that for $t<\theta$,
\[
\sum_{j=1}^d \int_0^t ~^T h ~ \left( \Phi_s^{\ast} [V_j , V_i ]
\right) (x_0) \circ dP^j_{s,T},~~i=1,...,d.
\]
is constant. Therefore, for $0 \leq s <\theta$,
\[
^T h ~ \left( \Phi_s^{\ast} [V_j , V_i ] \right)
(x_0)=0,~~i,j=1,...,d.
\]
By applying this at $s=0$, we obtain then
\[
^T h ~ [V_j , V_i ](x_0)=0,~~i,j=1,...,d.
\]
New iterations of the It\^o's formula show then that, we actually
have
\[
^T h ~U(x_0)=0,~~U \in \mathfrak{L}^2(x_0),
\]
so that $h=0$.

We mention that the problem of the existence of densities for stochastic differential equations driven by Brownian bridges was also discussed in \cite{CF}. But unlike our case, the existence of the density is discussed for times $t<T$.



\section{Free Carnot Groups and H\"ormander's Type Theorems}
In this section we now state some basic facts about Carnot groups.

\subsection{Free Carnot groups}
We introduce the notion of Carnot groups.
Such Lie groups appear as tangent spaces to hypoelliptic diffusions (see \cite{Ba2}).
For more details on the material presented in this section, we refer to the Chapter 2 of \cite{Baubook}.
Let $N \geq 1$. A Carnot group of depth (or step) $N$ is a simply
connected Lie group $\mathbb{G}$ whose Lie algebra can be written
\[
\mathcal{V}_{1}\oplus...\oplus \mathcal{V}_{N},
\]
where
\[
\lbrack \mathcal{V}_{i},\mathcal{V}_{j}]=\mathcal{V}_{i+j}
\]
and
\[
\mathcal{V}_{s}=0,\text{ for }s>N.
\]
\begin{example}[Heisenberg Group]
The Heisenberg group\index{Heisenberg group} $\mathbb{H}$ can be
represented as the set of $3\times3$ matrices:
\[
\left(
\begin{array}
[c]{ccc}
1 & x & z\\
0 & 1 & y\\
0 & 0 & 1
\end{array}
\right)  ,\text{ \ }x,y,z\in\mathbb{R}.
\]
The Lie algebra of $\mathbb{H}$ is spanned by the matrices
\[
D_{1}=\left(
\begin{array}
[c]{ccc}
0 & 1 & 0\\
0 & 0 & 0\\
0 & 0 & 0
\end{array}
\right)  ,\text{ }D_{2}=\left(
\begin{array}
[c]{ccc}%
0 & 0 & 0\\
0 & 0 & 1\\
0 & 0 & 0
\end{array}
\right)  \text{ and }D_{3}=\left(
\begin{array}
[c]{ccc}
0 & 0 & 1\\
0 & 0 & 0\\
0 & 0 & 0
\end{array}
\right)  ,
\]
for which the following equalities hold
\[
\lbrack D_{1},D_{2}]=D_{3},\text{ }[D_{1},D_{3}]=[D_{2},D_{3}]=0.
\]
Thus
\[
\mathfrak{h} \sim \mathbb{R} \oplus [\mathbb{R},\mathbb{R}],
\]
and, therefore, $\mathbb{H}$ is a (free) two-step Carnot group.
\end{example}

Let us now take a basis $U_1,...,U_d$ of the vector space
$\mathcal{V}_1$. The vectors $U_i$'s can be seen as left invariant
vector fields on $\mathbb{G}$ so that we can consider the
following stochastic differential equation on $\mathbb{G}$:
\begin{equation}
\label{SDEcarnot} d\tilde{B}_t =\sum_{i=1}^d \int_0^t U_i
(\tilde{B}_s) \circ dB^i_s, \text{ } t \geq 0,
\end{equation}
where $(B_t)_{t \geq 0}$ is a standard Brownian motion. This
equation is easily seen to have a unique (strong) solution
$(\tilde{B}_t)_{t \geq 0}$ associated with the initial condition
$\tilde{B}_0=0_{\mathbb{G}}$.

\begin{definition}
The process $(\tilde{B}_t)_{  t \geq 0}$ is called the lift of the
standard Brownian motion $(B_{t})_{  t \geq 0}$ in the group
$\mathbb{G}$ with respect to the basis $(U_1,...,U_d)$.
\end{definition}
Notice that $(\tilde{B}_t)_{  t \geq 0}$ is a Markov process with
generator $\frac{1}{2} \sum_{i=1}^d U_i ^2$. This second-order
differential operator is, by construction, left-invariant and
hypoelliptic. For $(\tilde{B}_t)_{  t \geq 0}$, we actually have
an explicit expression. To give this expression, we first have to
introduce some notations. If $I=(i_1,...,i_k) \in \{ 1,..., d
\}^k$ is a word, we denote $\mid I \mid=k $ its length and by
$U_I$ the commutator defined by
\[
U_I = [U_{i_1},[U_{i_2},...,[U_{i_{k-1}}, U_{i_{k}}]...].
\]
The group of permutations of the index set $\{1,...,k\}$ is
denoted $\mathfrak{S}_k$. If $\sigma \in \mathfrak{S}_k$, we
denote $e(\sigma)$ the cardinality of the set
\[
\{ j \in \{1,...,k-1 \} , \sigma (j) > \sigma(j+1) \}.
\]
As a direct consequence of the Chen-Strichartz development theorem
(see Proposition 2.3 of \cite{Baubook}, or \cite{Che}, \cite{Stri}), we
have
\begin{proposition}
We have
\[
\tilde{B}_t = \exp \left( \sum_{k = 1}^N \sum_{I=(i_1,...,i_k)}
\Lambda_I (B)_t U_I \right), \text{ }t \geq 0,
\]
where:
\[
\Lambda_I (B)_t=\sum_{\sigma \in \mathfrak{S}_k} \frac{\left(
-1\right) ^{e(\sigma )}}{k^{2}\left(
\begin{array}{l}
k-1 \\
e(\sigma )
\end{array}
\right) } \int_{0 \leq t_1 \leq ... \leq t_k \leq t} \circ
dB^{\sigma^{-1}i_1}_{t_1} \circ ... \circ
dB^{\sigma^{-1}i_k}_{t_k}.
\]
\end{proposition}
For instance:
\begin{enumerate}
\item The component of the process $(\ln (\tilde{B}_t))_{t \geq 0}$  in $\mathcal{V}_1$ is simply
\begin{equation*}
\sum_{i=1}^d U_i B^i_t , \text{ } t \geq 0.
\end{equation*}
\item The component in $\mathcal{V}_2$ is
\begin{equation*}
\frac{1}{2} \sum_{1 \leq i<j \leq d} [ U_i , U_j ] \left( \int_0^t
B^i_s  dB^j_s-B^j_s dB^i_s \right), \text{ } t \geq 0.
\end{equation*}
\end{enumerate}
The Carnot group $\mathbb{G}$ is said to be free if $\mathfrak{g}$
is isomorphic to the free Lie algebra with $d$ generators with the
relations that all brackets of length more than $N$ vanish . In
that case, $\dim \mathcal{V}_{j}$ is the number of Hall words of
length $j$ in the free algebra with $d$ generators. We thus have,
according to Bourbaki \cite{Bour} (see also Reutenauer \cite{Reu}
pp.96):
\begin{equation*}
\dim \mathcal{V}_{j}= \frac{1}{j} \sum_{i \mid j} \mu (i)
d^{\frac{j}{i}}, \text{ } j \leq N,
\end{equation*}
where $\mu$ is the M\"obius function. We easily deduce from this
that when $N \rightarrow +\infty$,
\[
\dim \mathfrak{g} \sim \frac{d^N}{N}.
\]
An important algebraic point is that, up to an isomorphism there
is one and only one free Carnot with a given depth and a given
dimension for the basis. Let us denote $m=\dim \mathbb{G}$.
Choose now a Hall family and consider the $\mathbb{R}^m$-valued
process $(B_t^{\ast})_{t \geq 0}$ obtained by writing the
components of $(\ln (\tilde{B}_t))_{t \geq 0}$ in the
corresponding Hall basis of $\mathfrak{g}$. It is easily seen that
$(B_t^{\ast})_{t \geq 0}$ solves a stochastic differential
equation that can be written
\[
B_t^{\ast}=\sum_{i=1}^d \int_0^t D_i (B_s^{\ast}) \circ dB^i_s,
\]
where the $D_i$'s are polynomial vector fields on $\mathbb{R}^m$
(for an explicit form of the $D_i$'s, which depend of the choice
of the Hall basis, we refer to Vershik-Gershkovich \cite{Ger-Ver}
pp.27) . With these notations, we have the following proposition
(see also \cite{Ger-Ver}).
\begin{proposition}
\label{Def Carnot numerique} On $\mathbb{R}^m$, there exists a
unique group law $\star$ which makes the vector fields
$D_1,...,D_d$ left invariant. This group law is polynomial of
degree $N$ and we have moreover
\[
\left( \mathbb{R}^m , \star \right) \sim \mathbb{G}.
\]
The group $\left( \mathbb{R}^m , \star \right)$ is called
the free Carnot group of step $N$ over $\mathbb{R}^d$. It shall be
denoted $\mathbb{G}_{N} (\mathbb{R}^d )$. The process $B^{\ast}$
shall be called the lift of $B$ in $\mathbb{G}_{N} (\mathbb{R}^d
)$.
\end{proposition}
\begin{remark}
Notice that $\mathbb{G}_{N} (\mathbb{R}^d )$ is, by construction,
endowed with the basis of vector fields $(D_1,...,D_d)$. These
vector fields agree at the origin with
$\left(\frac{\partial}{\partial
x_1},\cdots,\frac{\partial}{\partial x_d} \right)$.
\end{remark}
\subsection{H\"ormander's type theorems}
Consider now on $\mathbb{R}^n$ stochastic differential equations
of the type
\begin{equation}  \label{SDE}
X_t =x_0 + \sum_{i=1}^d \int_0^t V_i (X_s) \circ dB^i_s, \text{ }
t \geq 0,
\end{equation}
where:

\begin{enumerate}
\item  $x_{0}\in \mathbb{R}^{n}$;

\item  $V_{1},...,V_{d}$ are $C^{\infty }$ bounded vector fields on $%
\mathbb{R}^{n}$;

\item  $\circ$ denotes Stratonovitch integration;

\item  $(B^1_{t},...,B^d_{t})_{ t \geq 0}$ is a $d$-dimensional
standard Brownian motion.
\end{enumerate}

It is well-known that for every $x_0 \in \mathbb{R}^n$, there is a
unique solution $(X_t^{x_0})_{t \geq 0}$ to (\ref{SDE}) and
moreover that there exists a stochastic flow $(\Phi_t, t \geq 0)$
of smooth diffeomorphisms $\mathbb{R}^n \rightarrow \mathbb{R}^n$
associated to the equations (\ref{SDE}). Let us denote by
$\mathfrak{L}$ the Lie algebra generated by the vector fields
$V_i$ and for $p \geq 2$, by $\mathfrak{L}^p$ the Lie subalgebra
defined by
\[
\mathfrak{L}^p = \{ [X,Y], \text{ } X \in \mathfrak{L}^{p-1}, Y
\in \mathfrak{L} \}.
\]
Moreover if $\mathfrak{a}$ is a subset of $\mathfrak{L}$, we
denote
\[
\mathfrak{a} (x) = \{ V(x) , V \in \mathfrak{a} \}, \text{ }x \in
\mathbb{R} ^n.
\]
In this framework, we have the following:

\begin{theorem}
\label{Hormander itéré} Let $x_0 \in \mathbb{R}^n$. If
$\mathfrak{L}^{N+1}(x_0) =\mathbb{R}^n$, then for any $t>0$, the
random variable
\[
(X_t^{x_0},B_t^{\ast})
\]
has a smooth density with respect to any Lebesgue measure on
$\mathbb{R}^n \times \mathbb{G}_{N} (\mathbb{R}^d )$, where
$(X_t^{x_0})_{t \geq 0}$ is the solution of ( \ref{SDE}) with
initial condition $x_0$ and $(B_t^{\ast})_{t\geq 0}$ the lift of
$(B_t)_{t\geq 0}$ in $\mathbb{G}_{N} (\mathbb{R}^d )$.
\end{theorem}

\begin{proof}
With a slight abuse of notation, we still denote  $V_i$ (resp.
$D_i$) the extension of $V_i$ (resp. $D_i$) to the space
$\mathbb{R}^n \times \mathbb{G}_{N} (\mathbb{R}^d )$. The process
$(X_t^{x_0},B_t^{\ast})_{t \geq 0}$ is easily seen to be a
diffusion process in $\mathbb{R}^n \times \mathbb{G}_{N}
(\mathbb{R}^d )$ with infinitesimal generator
\begin{equation*}
\frac{1}{2} \sum_{i=1}^d ( V_i +D_i)^2.
\end{equation*}
Thus, to prove the theorem, it is enough to check the
H\"ormander's condition for this operator at the point $(x_0,0)$.
Now, notice that  $[ \mathfrak{L} , \mathfrak{g}_{N} (\mathbb{R}^d
)]=0$, so that
\begin{equation*}
\mathbf{Lie} ( V_1 + D_1 ,..., V_n +D_n)(x_0,0) \simeq
\mathfrak{L}^{N+1}(x_0) \oplus \mathfrak{g}_{N} (\mathbb{R}^d ),
\end{equation*}
because $\mathfrak{g}_{N} (\mathbb{R}^d )$ is step $N$ nilpotent.
We denoted $\mathbf{Lie} ( V_1 + D_1 ,..., V_n +D_n)$ the Lie
algebra generated by $(V_1+D_1,...,V_n+D_n)$. The conclusion
follows readily.
\end{proof}

\begin{example}
For $N=0$, we have $\mathbb{G}_{0} (\mathbb{R}^d )= \{ 0 \} $ and
Theorem \ref{Hormander itéré} is the classical H\"ormander's
theorem.
\end{example}

\begin{example}
For $N=1$, we have $\mathbb{G}_{1} (\mathbb{R}^d ) \simeq
\mathbb{R}^d$ and Theorem \ref {Hormander itéré} gives a sufficent
condition for the existence of a smooth density for the variable
\[
(X_t^{x_0},B_t).
\]
\end{example}

\begin{example}
For $N=2$, we have $\mathbb{G}_{2} (\mathbb{R}^d ) \simeq
\mathbb{R}^d \times \mathbb{R}^{ \frac{d(d-1)}{2}}$ and Theorem
\ref{Hormander itéré} gives a sufficient condition for the
existence of a smooth density for the variable
\[
(X_t^{x_0},B_t, \wedge B_t).
\]
where
\[
\wedge B_t = \left( \frac{1}{2} \int_0^t B^i_s dB^j_s-B^j_s dB^i_s
\right)_{1 \leq i < j \leq d}.
\]
\end{example}

\section{Stochastic Differential Equations Driven by N-Step Brownian Loops}

In this section, we now enter into the heart of our study.

\subsection{N-step Brownian Loops}
On the free Carnot group $\mathbb{G}_{N} (\mathbb{R}^d )$,
consider the fundamental process $(B_t^{\ast})_{t \geq 0}$ defined
as the solution of the stochastic differential equation
\[
B_t^{\ast} =\sum_{i=1}^d \int_0^t D_i (B_s^{\ast}) \circ dB^i_s,
\text{ } t \geq 0.
\]
As a consequence of H\"ormander's theorem, the diffusion with
generator
\[
\frac{1}{2} \sum_{i=1}^d D_i^2
\]
has a smooth transition kernel $p_t (x,y)$, $t>0$ with respect to
the Lebesgue measure.

\begin{proposition}
Let $T>0$. There exists a unique $\mathbb{R}^d $-valued continuous
process $(P^N_{t,T})_{0 \leq t \leq T}$ such that
\begin{equation}
\label{N step loop} P^{N,i}_{t,T}= B^i_t + \int_0^t  D_i \ln
p_{T-s} \left( P^{N,*}_{s,T},0_{\mathbb{G}_{N} (\mathbb{R}^d )}
\right) ds , ~~t<T,~~i=1,...,d,
\end{equation}
where $(P^{N,\ast}_{t,T})_{0 \leq t \leq T}$ denotes the lift of $(P^N_{t,T})_{0 \leq t \leq T}$ in $\mathbb{G}_{N} (\mathbb{R}^d )$.
It enjoys the following properties:

\begin{enumerate}
\item  $P^{N,*}_{T,T}=0_{\mathbb{G}_{N} (\mathbb{R}^d )}$, almost surely;

\item  for any predictable and bounded functional $F$,
\[
\mathbb{E} \left( F \left( (B_t)_{0 \leq t \leq T} \right) \mid
B_T^{\ast} =0_{\mathbb{G}_{N} (\mathbb{R}^d )} \right)=\mathbb{E}
\left( F \left( (P^N_{t,T})_{0 \leq t \leq T} \right) \right);
\]
\item  $(P^N_{t,T})_{0 \leq t \leq T}$ is a semimartingale up to time $T$.
\end{enumerate}
\end{proposition}

\begin{proof}
The construction of the bridge over a given diffusion is very
classical and very general (see for example \cite{Ba}, \cite{Bi},
\cite{Fi-Pi-Yo} and \cite{Hs} p.142), so that we do not present
the details. The only delicate point in the previous statement is
the semimartingale property up to time $T$. In the elliptic case
Bismut \cite{Bi} deals with the end point singularity by proving
an estimate of the logarithmic derivatives of the heat kernel. For
the heat kernel on Carnot groups, such an estimate can directly be
obtained from \cite{Jer-San} and \cite{Bon}. Namely, for $g \in
\mathbb{G}_N (\mathbb{R}^d)$, $t>0$,
\begin{equation*}
\mid D_{i} \ln p_t (g,0) \mid \leq \frac{C}{\sqrt{t}}
\end{equation*}
where $C>0$.
Now, to prove that $(P^{N,*}_{t,T})_{0
\leq t \leq T}$ is a semimartingale up to time $T$, we need to
show that for any $1 \leq i \leq d$,
\begin{equation*}
\int_0^T \mid D_i \ln p_{T-s} (P^{N,*}_{s,T},0) \mid ds < + \infty
\end{equation*}
with probability 1, which follows therefore from the above
estimate.

\end{proof}

The semimartingale $(P^N_{t,T})_{0
\leq t \leq T}$  shall be called a Brownian loop of step $N$.

\begin{example}
The process $(P^1_{t,T})_{0 \leq t \leq T}$ is simply the
$d$-dimensional Brownian bridge from $0$ to $0$ with length $T$.
\end{example}

\begin{example}
The process $(P^2_{t,T})_{0 \leq t \leq T}$ is the $d$-dimensional
standard Brownian motion $(B_t)_{0 \leq t \leq T}$ conditioned by
$(B_T , \wedge B_T)=0$.
\end{example}

\begin{remark}
Notice that in law,
\begin{equation}  \label{scaling}
( P^N_{t,T})_{0 \leq t \leq T}=(\sqrt{T} P^N_{\frac{t}{T},1})_{0
\leq t \leq T}.
\end{equation}
\end{remark}
\subsection{SDEs Driven by N-Step
Brownian Loops}

Consider now on $\mathbb{R}^n$ stochastic differential equations
of the type
\begin{equation}  \label{SDEbridgestep}
X_t =x_0 + \sum_{i=1}^d \int_0^t V_i (X_s) \circ dP^{i,N}_{s,T},
\text{ } t \leq T
\end{equation}
where:

\begin{enumerate}
\item  $x_0 \in \mathbb{R}^n$;

\item  $V_1,...,V_d$ are $C^{\infty}$ bounded vector fields on $\mathbb{R}^n$
;

\item  $(P^{1,N}_{t,T},...,P^{d,N}_{t,T})_{0 \leq t \leq T}$ is a $d$
-dimensional $N$-step Brownian loop from $0$ to $0$ with length
$T>0$.
\end{enumerate}

\begin{proposition}
For every $x_0 \in \mathbb{R}^n$, there is a unique solution
$(X_t^{x_0})_{0 \leq t \leq T}$ to (\ref{SDEbridgestep}). Moreover
there exists a stochastic flow $(\Phi_t, 0 \leq t \leq T)$ of
smooth diffeomorphisms $\mathbb{R}^n \rightarrow \mathbb{R}^n$
associated to the equations (\ref{SDEbridgestep}).
\end{proposition}

We consider now the following family of operators
$(\mathcal{H}^N_T)_{T \geq 0}$ defined on the space of compactly
supported smooth functions $f:\mathbb{R}^n \rightarrow \mathbb{R}$
by
\[
(\mathcal{H}^N_T f)(x) =\mathbb{E}\left( f(X_{T}^{x})\right),
\text{ }x \in \mathbb{R}^n.
\]
The operator $\mathcal{H}^N_T$ shall be called the depth $N$
holonomy operator.
\begin{remark}
Of course, as for $N=1$ which is the case treated in section 2,
the operator $\mathcal{H}_T^{N}$ does not satisfy a semi-group
property.
\end{remark}

A relevant intrinsic property of $\mathcal{H}^N_T$ is the following:

\begin{proposition}
$\mathcal{H}^N_T$ does not depend on the particular free Carnot group $\mathbb{G}_N (\mathbb{R}^d)$.
\end{proposition}

\begin{proof}
Consider the stochastic differential equation
\begin{equation}  \label{SDEbridgestep}
X_t =x_0 + \sum_{i=1}^d \int_0^t V_i (X_s) \circ d\tilde{P}^{i,N}_{s,T},
\text{ } t \leq T
\end{equation}
where $(\tilde{P}^{1,N}_{t,T},...,\tilde{P}^{d,N}_{t,T})_{0 \leq t \leq T}$ generates a loop in a free Carnot group $\mathbb{G}$ of step N.
Let $(B_t)_{t \geq 0}$ denote a $d$-dimensional standard Brownian motion, and let $(B^*_t)_{t \geq 0}$
(resp. $(\tilde{B}_t)_{t \geq 0}$) denote a lift in  $\mathbb{G}_N (\mathbb{R}^d)$ (resp. $\mathbb{G}$).
Thanks to the proposition 2.10 of \cite{Baubook}, there exists a Carnot group isomorphism
\[
\phi: \mathbb{G}_N (\mathbb{R}^d)  \rightarrow \mathbb{G}
\]
such that $\tilde{B}=\phi (B^*)$.
Therefore almost surely, $\tilde{B}_T=0$ is equivalent to $B^*_T=0$ and thus
\[
(\tilde{P}^{1,N}_{t,T},...,\tilde{P}^{d,N}_{t,T})_{0 \leq t \leq T}=^{law}(P^{1,N}_{t,T},...,P^{d,N}_{t,T})_{0 \leq t \leq T}.
\]
\end{proof}

\begin{theorem}
\label{Theogeneral} Let $f:\mathbb{R}^n \rightarrow \mathbb{R}$ be
a smooth, compactly supported function. In $\mathbf{L}^2$,
\[
\lim_{T \rightarrow 0} \frac{\mathcal{H}^{N}_T f -f}{T^{N+1}}
=\Delta_{N} f,
\]
where $\Delta_{N}$ is a second order differential operator.
\end{theorem}

\begin{proof}
Before we start the proof, let us precise some notations we
already used. If $I=(i_1,...,i_k) \in \{ 1,..., d \}^k$ is a word,
we denote $\mid I \mid=k $ its length and by $V_I$ the commutator
defined by
\[
U_I = [U_{i_1},[U_{i_2},...,[U_{i_{k-1}}, U_{i_{k}}]...].
\]
The group of permutations of the index set $\{1,...,k\}$ is
denoted $\mathfrak{S}_k$. If $\sigma \in \mathfrak{S}_k$, we
denote $e(\sigma)$ the cardinality of the set
\[
\{ j \in \{1,...,k-1 \} , \sigma (j) > \sigma(j+1) \}.
\]
Finally, we denote
\[
\Lambda_I (P^N_{.,T})_t=\sum_{\sigma \in \mathfrak{S}_k}
\frac{\left( -1\right) ^{e(\sigma )}}{k^{2}\left(
\begin{array}{l}
k-1 \\
e(\sigma )
\end{array}
\right) } \int_{0 \leq t_1 \leq ... \leq t_k \leq t} \circ
dP^{N,\sigma^{-1}i_1}_{t_1,T} \circ ... \circ
dP^{N,\sigma^{-1}i_k}_{t_k,T}.
\]
Due to the scaling property
\begin{equation*}
( P^N_{t,T})_{0 \leq t \leq T}=(\sqrt{T} P^N_{\frac{t}{T},1})_{0
\leq t \leq T},
\end{equation*}
we can closely follow the article of Strichartz \cite{Stri} (see
also   Castell \cite{Cast}), to obtain
the following asymptotic development of $f(X_T^x)$:
\[
f(X_T^x)=\left( \exp \left( \sum_{k = 1}^{2N+2}
\sum_{I=(i_1,...,i_k)} \Lambda_I (P^N_{.,T})_T V_I \right) f
\right) (x) +T^{\frac{2N+3}{2}} \mathbf{R}_{2N+3} (T,f,x),
\]
where the remainder term satisfies when $T \rightarrow 0$,
\[
\sup_{x\in\mathbb{R}^{n}}\sqrt{\mathbb{E} \left( \mathbf{R}_{2N+3}
(T,f,x)^{2}\right)}\leq C\,
\]
for some non negative constant $C$. By definition of $(
P^N_{t,T})_{0 \leq t \leq T}$, we actually have
\[
\sum_{k = 1}^{2N+2} \sum_{I=(i_1,...,i_k)} \Lambda_I (P^N_{.,T})_T
V_I  =\sum_{k = N+1}^{2N+2}\sum_{I=(i_1,...,i_{k})} \Lambda_I
(P^N_{.,T})_T V_I,
\]
so that
\[
f(X_T^x)=\left( \exp \left( \sum_{k =
N+1}^{2N+2}\sum_{I=(i_1,...,i_{k})} \Lambda_I (P^N_{.,T})_T V_I
\right) f \right) (x) +T^{\frac{2N+3}{2}} \mathbf{R}_{2N+3}
(T,f,x).
\]
Therefore, since $f$ is assumed to be compactly supported
\[
\mathcal{H}_T^N f (x)=\mathbb{E} \left( \left( \exp \left( \sum_{k
= N+1}^{2N+2} \sum_{I=(i_1,...,i_{k})} \Lambda_I (P^N_{.,T})_T V_I
\right) f \right) (x) \right) +T^{\frac{2N+3}{2}}
\tilde{\mathbf{R}}_{2N+3} (T,f,x),
\]
where
\[
\tilde{\mathbf{R}}_{2N+1} (T,f,x) = \mathbb{E} \left(
\mathbf{R}_{2N+1} (T,f,x) \right).
\]
Since, by symmetry, we always have
\[
\mathbb{E} \left( \Lambda_I (P^N_{.,T})_T  \right)=0,
\]
we have to go at the order 2 in the asymptotic development of the
exponential when $T \rightarrow 0$. By neglecting the terms which
have order more than $T^{\frac{2N+3}{2}}$, we obtain
\[
\mathcal{H}_T^N f (x) =f(x)+ \sum_{
\begin{array}{l}
I=(i_{1},...,i_{N+1}) \\
J=(j_{1},...,j_{N+1})
\end{array}
} \frac{1}{2} \mathbb{E} \left( \Lambda_I (P^N_{.,T})_T \Lambda_J
(P^N_{.,T})_T \right) (V_I V_J f)(x)
+T^{\frac{2N+3}{2}}\mathbf{R}_{2N+3}^{\ast}(T,f,x),
\]
where the remainder term $\mathbf{R}_{2N+3}^{\ast}(T,f,x)$ is
bounded in $\mathbf{L}^2$ when $T \rightarrow 0$. This leads to the expected result.
\end{proof}

\begin{example}
We have
\[
\Delta_0=\frac{1}{2} \sum_{i=1}^d V_i^2,
\]
and, as already seen in section 2,
\[
\Delta_1=\frac{1}{24} \sum_{1 \leq i < j \leq d} [V_i,V_j]^2.
\]
\end{example}

We now have the following generalization of Theorem \ref{Holonomie
step2} and Theorem \ref{Hormander step2}:

\begin{theorem}
\label{Theoreme general} Assume that $\mathfrak{L}^{N+1} = 0$,
then for any solution $(X_t^{x_0})_{0 \leq t \leq T}$ of
(\ref{SDEbridgestep}) we have
almost surely $X_T^{x_0} = x_0$. On the other hand, assume that $\mathfrak{L}%
^{N+1} (x_0)=\mathbb{R}^n$, then for the solution $(X_t^{x_0})_{0
\leq t \leq T}$ of (\ref{SDEbridgestep}) the random variable
$X_T^{x_0}$ has a smooth density with respect to the Lebesgue
measure of $\mathbb{R}^n$.
\end{theorem}

\begin{proof}
Assume that $\mathfrak{L}^{N+1} = 0$, then there exists a smooth
map
\[
F: \mathbb{R}^n \times \mathbb{G}_{N} (\mathbb{R}^d) \rightarrow
\mathbb{R}^n
\]
such that, for $x_0 \in \mathbb{R}^n$, the solution
$(X_t^{x_0})_{0 \leq t \leq T}$ of the SDE (\ref{SDEbridgestep})
can be written
\[
X_t^{x_0}=F(x_0, Q^N_{t,T} ),
\]
which implies immediately the expected result.

Assume now that $\mathfrak{L} ^{N+1} (x_0)=\mathbb{R}^n$. Let us
consider  the solution $(Z_t)_{t \geq 0}$ of the following
stochastic differential equation:
\begin{equation*}
Z_t =x_0 + \sum_{i=1}^d \int_0^t V_i (Z_s) \circ dB^i_s, \text{ }
t \geq 0,
\end{equation*}
where $(B^1_{t},...,B^d_{t})_{ t \geq 0}$ is a $d$-dimensional
standard  Brownian motion. From Theorem \ref{Hormander itéré}, the
random variable
\begin{equation*}
(Z_T,B_T^{\ast})
\end{equation*}
has a smooth density with respect to any Lebesgue measure on
$\mathbb{R}^n \times \mathbb{G}_{d,N}$. It implies in the same way
as in the proof of Theorem \ref{Hormander step2} that $X_T^{x_0}$
has a density with respect to the Lebesgue measure because the
density of $B_T^{\ast}$ does not vanish at $0$ (see \cite{Ben}).

\end{proof}

In the case of the existence of a density for $X_T^{x_0}$, we can
moreover give an equivalent of this density when the length of the
loop tends to 0. To this end, let us precise some notations.

We set for $x \in \mathbb{R}^n$ and $k \geq N$,
\[
\mathcal{U}_k (x)= \mathbf{span} \{ V_I , \text{ } N \leq \mid I
\mid \leq k \}.
\]
In the case where $\mathfrak{L}^{N+1} (x)=\mathbb{R}^n$, if $k$ is
big enough then $\mathcal{U}_k (x)= \mathbb{R}^n$. We denote
$d(x)$ the smallest integer $k \geq N+1$ for which this equality
holds and define the graded dimension
\[
\dim_{\mathcal{H}} \mathfrak{L}^{N+1} (x) :=\sum_{k=N+1}^{d(x)} k
\left( \dim \mathcal{U}_{k} (x) - \dim \mathcal{U}_{k-1}(x)
\right).
\]

\begin{theorem}
Assume that for any $x \in \mathbb{R}^n$, $\mathfrak{L}^{N+1}
(x)=\mathbb{R} ^n$. Let us denote $p_T(x)$ the density of $X_T^x$
with respect to the Lebesgue measure. We have
\[
p_T (x) \sim_{T \rightarrow 0} \frac{m (x)}{T^{\frac{\dim_{\mathcal{H}} %
\mathfrak{L}^{N+1} (x)}{2}}},
\]
where $m$ is a smooth non negative function.
\end{theorem}

\begin{proof}
Let us, once time again, consider  the solution $(Z_t^x)_{t \geq
0}$ of the following stochastic differential equation:
\begin{equation*}
Z_t^x =x + \sum_{i=1}^d \int_0^t V_i (Z_s^x) \circ dB^i_s, \text{
} t \geq 0,
\end{equation*}
where $(B^1_{t},...,B^d_{t})_{ t \geq 0}$ is a $d$-dimensional
standard  Brownian motion. From \cite{Ben} (see also \cite{Le}) ,
the density at $(x,0)$ of the random variable $(Z_T^x,B_T^{\ast})$
behaves when $T$ goes to zero like
\begin{equation*}
\frac{\tilde{m} (x)}{T^{ \frac{ \dim_{\mathcal{H}} \mathbb{G}_{N}
(\mathbb{R}^d)+ \dim_{\mathcal{H}} \mathfrak{L}^{N+1} (x) }{2} }},
\end{equation*}
where $\dim_{\mathcal{H}}\mathbb{G}_{N}
(\mathbb{R}^d)=\sum_{j=1}^N j \dim \mathcal{V}_j$ is the graded
dimension of $\mathbb{G}_{N} (\mathbb{R}^d)$, and $\tilde{m}$ a
smooth non negative function. Always from \cite{Ben}, the density
of the random variable $Y_T$ behaves when $T$ goes to zero like
\begin{equation*}
\frac{C}{T^{ \frac{ \dim_{\mathcal{H}} \mathbb{G}_{N}
(\mathbb{R}^d) }{2} }},
\end{equation*}
where $C$ is a non negative constant. The conclusion follows
readily.
\end{proof}

\end{document}